\documentclass[10pt]{article}
\usepackage{amssymb}
\usepackage{amsthm}
\usepackage{amsmath}
\usepackage{mathrsfs}
\usepackage{verbatim}

\begin{document}

\title{Multi-marginal optimal transport and multi-agent matching problems: uniqueness and structure of solutions\footnote{The author is pleased to acknowledge the support of a University of Alberta start-up grant and a National Sciences and Engineering Research Council of Canada Discovery Grant.}}
\author{Brendan Pass\footnote{Department of Mathematical and Statistical Sciences, 632 CAB, University of Alberta, Edmonton, Alberta, Canada, T6G 2G1 pass@ualberta.ca.}}
\maketitle

\begin{abstract}
We prove uniqueness and Monge solution results for multi-marginal optimal transportation problems with a certain class of surplus functions; this class arises naturally in multi-agent matching problems in economics.  This result generalizes a seminal result of Gangbo and \'Swi\c{e}ch on multi-marginal problems.  Of particular interest, we show that this also yields a partial generalization of the Gangbo-\'Swi\c{e}ch result to manifolds; alternatively, we we can think of this as a partial extension of McCann's theorem for quadratic costs on manifolds to the multi-marginal setting.
   
We also show that the class of surplus functions considered here neither contains, nor is contained in, the class of surpluses studied in \cite{P1}, another generalization of Gangbo and \'Swi\c{e}ch's result.
\end{abstract}
\maketitle
\section{Introduction}
Two marginal optimal transportation is the general problem of coupling two distributions of mass together as efficiently as possible, relative to a given surplus function.  This is an exciting and very active area of research, with a wide variety of applications.  For a full literature review, see the book of Villani \cite{V2}.

Recently, optimal transport problems with several marginals have begun to attract more attention, due to emerging applications in economics \cite{CarEke} \cite{cmn}, mathematical finance \cite{ght}\cite{hl}\cite{bhlp}, condensed matter physics \cite{cfk} \cite{bdpgg}, barycenters and image processing \cite{ac} \cite{bdpr} and the characterization of $m$-cyclically monotone vector fields \cite{GG}.  This is the general problem of aligning \textit{several} mass distributions with maximal efficiency.  Stated precisely, given Borel probability measures $\mu_i$ on bounded, open sets $M_i \subseteq \mathbb{R}^n$, for $i=1,2...m$ and a surplus function $b:M_1 \times M_2 \times...\times M_m \rightarrow \mathbb{R}$, the problem is to maximize

\begin{equation} \tag{MK}\label{MK}
\int_{M_1 \times M_2 \times...\times M_m} b(x_1,x_2,...,x_m)d\gamma
\end{equation}
among measures $\gamma$ on $M_1 \times M_2 \times...\times M_m$ which project to the $\mu_i$.  Under reasonable conditions, a maximizer $\gamma$ exists.  We say that $\gamma$ is a \textit{Monge solution} if it is concentrated on the graph of a function over the first variable.

When $m=2$, this is the classical, two marginal Monge-Kantorovich problem.  In that setting, assuming the first marginal is absolutely continuous with respect to Lebesgue measure and the surplus function $b$ satisfies a fairly weak condition, called the twist condition, it is by now well known that the solution $\gamma$ induces a Monge solution and is unique \cite{lev}\cite{GM}\cite{caf}.

When $m>2$, however, the picture is much muddier.   In contrast to the two marginal case, the structure of solutions to these problems is not well understood, except in certain special cases.  Uniqueness and Monge solution results have been proven by Gangbo and \'Swi\c{e}ch \cite{GS} for the quadratic surplus, 
\begin{equation}\label{gss}
b(x_1,x_2,...,x_m)=-\sum_{i\neq j}^m|x_i-x_j|^2; 
\end{equation}
see also \cite{OR}\cite{KS}\cite{RU}\cite{RU2}.  This was generalized by Heinich  \cite{H} to surplus functions of the form 
\begin{equation}\label{hs}
b(x_1,x_2,...,x_m) =h(\sum_{i=1}^mx_i),
\end{equation}
where $h:\mathbb{R}^n \rightarrow \mathbb{R}$ is strictly convex.  When $n=1$, Carlier was able to further generalize this result to the class of strictly $2$-monotone surplus functions. In higher dimensions, a result of the present author  establishes uniqueness and Monge solutions under strong second order conditions on the surplus \cite{P1}, encompassing the results in \cite{GS} and \cite{H}.  In contrast to the two marginal case, these results depend strongly on the form of the surplus function; a series of counterexamples in \cite{CN} \cite{P} and \cite{P6} indicate that solutions may not be unique or induce Monge solutions, even when the surplus functions satisfy naive multi-marginal generalizations of the twist condition.  In fact, the examples in \cite{P} are supported on submanifolds $S \subseteq M_1 \times M_2 \times...\times M_m$ of dimension $n(m-1)$.

In this paper we study a multi-marginal optimal transportation problem for a particular class of surplus function.  Given some other open subset $Z\subseteq \mathbb{R}^n$ and functions $f_i:M_1 \times Z \rightarrow \mathbb{R}$, we will consider surplus functions of the form:

\begin{equation} \label{hp}
b(x_1,x_2,...x_m) = \sup_{z \in Z}\sum_{i=1}^mf_i(x_i,z)
\end{equation}

Surplus functions of this form arise naturally in multi-agent matching, or hedonic pricing, problems in mathematical economics.  These types of problems were originally introduced by Ekeland \cite{Eke} (when $m=2$) and Carlier and Ekeland \cite{CarEke} (when $m \geq 3$); the formulation as an optimal transport problem with a surplus of the form \eqref{hp} can be found in the paper of Chiappori, McCann and Nesheim \cite{cmn}.  Economically, each measure $\mu_i$ represents the distribution of a category of agents on the type space $M_i$ (parameterizing, for example, the agents' skill sets).  These distributions must be matched to make teams; each team consists of one agent from each category.  The matching is done via signing contracts $z \in Z$ and the functions $f_i(x_i,z)$ represents the preference of an agent of type $x_i$ for contract $z$.  The result of Chiappori, McCann and Nesheim tells us that finding an equilibrium in this market is equivalent to solving the optimal transport problem with surplus \eqref{hp}; the optimal measure $\gamma$ tells us which agents come together to form teams in equilibrium.
 
Our main goal in this paper is to resolve Monge solution and uniqueness questions for this class of surplus functions. Aside from its importance in economic applications, this class of surplus functions includes those studied by Gangbo and \'Swi\c{e}ch \eqref{gss} and Heinich \eqref{hs} (as we show in subsections 5.1 and 5.2), and so our main theorem can be seen as a generalization of their results.  Let us recall that the class of surplus functions treated in \cite{P1} also included \eqref{gss} and \eqref{hs}.  As we show in the subsections 5.3 and 5.4, our result here is neither strictly stronger \textit{nor} weaker than the result in \cite{P1}; that is, we exhibit explicit examples of surplus functions satisfying the hypotheses in \cite{P1} which are \textit{not} of the form \eqref{hp}, \textit{as well as} examples of the form \eqref{hp} which do \textit{not} satisfy the assumptions in \cite{P1}.  It is interesting to note that our assumptions on the functions $f_i$ (conditions \textbf{H1} -\textbf{H5} in the next section) look much more like typical optimal transportation type conditions than the strong, complicated looking conditions required in \cite{P1}.

Unlike the results in \cite{GS}\cite{H}\cite{C} and \cite{P1}, our argument here does not rely on a dual formulation of \eqref{MK}.  Instead, the proof of our main result (Theorem \ref{main}) exploits two main ingredients.  The first is a general, structural result on solutions to multi-marginal problems \cite{P}.  For surplus functions of the form \eqref{hp}, this result surprisingly implies that the support of the solution $n$-rectifiable and spacelike for a certain semi-Riemannian metric.  

The second ingredient is the interplay between \eqref{MK} and another variational problem, introduced by Carlier and Ekeland \cite{CarEke} (see \eqref{MAM} in the next section).  From an economic perspective, this second problem can be seen as another formulation of the multi-agent problem.  In this formulation, roughly speaking, one focuses on the couplings between each category of agents and the contracts they sign (rather than on  the coupling between the various categories of agents, as in the  formulation of Chiappori, McCann and Nesheim).  The problem is to find the probability measure $\nu$ on the space of contracts which maximizes the sum of the two marginal transportation surpluses with  each $\mu_i$; economically, this measure tells us the relative frequency of contracts that get executed in equilibrium and the optimal couplings between each $\mu_i$ and $\nu$ tell us which types in each category sign which contracts.

Roughly speaking, the rectifiability result from \cite{P} allow us to apply tools from geometric measure theory (namely, the co-area formula) to the correspondence between teams and contracts (stated precisely, this is a map $\overline{z}:spt(\gamma) \rightarrow Z$, pushing $\gamma$ forward to $\nu$).   We then use the spacelike condition to show that different teams $(x_1,x_2,...,x_m)$ in the support of  the optimizer always sign different contracts, at least locally, allowing us to deduce that $\nu$ is absolutely continuous with respect to Lebesgue measure.  From here, we can use standard, two marginal optimal transport techniques to establish that in equilibrium, under appropriate assumptions, almost all contracts $z \in Z$ are signed by at most one agent type $x_i:= F_i(z)\in M_i$ (which easily yields the uniqueness of $\gamma$), and that the map $F_1:M_1 \rightarrow Z$ is invertible almost everywhere.  We then have that, for $(x_1,x_2,...,x_m)$ in the support of the optimizer, $x_i = F_i(F_1^{-1}(x_1))$, implying that $\gamma$ induces a Monge solution.

As a consequence of our work here (in fact, as a direct ingredient in our argument) we prove a regularity result on the optimizer in the matching problem formulation of Carlier and Ekeland.  In addition to its economic implications, this corollary  may have independent interest.  In the special case when each $f_i(x_i,z) =-|x_i-z|^2$, the Carlier-Ekeland formulation is equivalent to finding the \textit{barycenter} of the measures $\mu_i$, a problem introduced and solved by Agueh and Carlier \cite{ac}.   Our regularity result generalizes a result by Agueh and Carlier, who use a correspondence between the multi-marginal problem with quadratic cost \eqref{gss} and the problem of finding a barycenter of several measures on $\mathbb{R}^n$; in that paper, they used this relationship to establish an $L^{\infty}$ estimate on the barycenter (assuming that at least one of the $\mu_i$ has an $L^{\infty}$ density).  At the same time, our approach can be seen as generalizing an argument of the present author in the two marginal case; in that work, the geometry of the optimal plan is used to establish regularity of $\nu$ \cite{P10}.    

It is also natural to consider the barycenter problem in the Wasserstein space over a Riemannian manifold.  Riemannian manifolds with non-positive sectional curvature fit naturally into the framework of this paper. We will show that  our result also implies a regularity result on the barycenter of measures on simply connected, non-positively curved Riemannian manifolds,  generalizing the regularity result of Agueh and Carlier on Euclidean space.  

Furthermore, in this setting, we show that our main result can be interpreted as a generalization of the theorem of Gangbo and \'Swi\c{e}ch to simply connected, negatively curved Riemannian manifolds, in the same spirit that McCann's polar factorization theorem on Riemannian manifolds generalizes Brenier's polar factorization theorem on Euclidean space.
 
In the next section, we will introduce our assumptions on the functions $f_i$, as well as recall Carlier and Ekeland's formulation of the optimal matching problem and key theorems from \cite{CarEke} and \cite{P} which we will need for the proof of our main result.  In the third section, we prove that the solution $\nu$ to Carlier and Ekeland's problem \eqref{MAM} is absolutely continuous with respect to Lebesgue measure.  In the fourth section, we prove our main result, while the fifth section is reserved for examples and applications.
\section{Preliminaries}
In this section, we present assumptions we will make on the functions $f_i$, and recall some key ingredients which will used in the proof.
\subsection{Definitions and assumptions on the $f_i$.}
Assuming $f_i \in C^2(M_i)$, we define $D^2_{x_iz}f_i(x_i,z)$ to be the $n \times n$ matrix of mixed, second order partial derivatives, 
\begin{equation*}
D^2_{x_iz}f_i(x_i,z)=(\frac{\partial^2c}{\partial x_i^{\alpha} z^{\beta}})_{\alpha\beta}. 
\end{equation*}
Other matrices of second order derivatives will be denoted analogously. We will say that $f_i$ is \textit{non-degenerate} if  $D^2_{x_iz}f_i(x_i,z)$ is invertible for all $(x_i,z) \in M_i \times Z$.

We will denote by $D_{x_i}f_i(x_i,z)$ (respectively $D_{z}f_i(x_i,z)$) the differential of $f_i$ with respect to $x_i$ (respectively $z$). We say $f_i$ is $(x_i,z)$-twisted (respectively $(z,x_i)$-twisted) if for all $x_i \in M_i$ (respectively $z \in Z$), the mapping 
\begin{equation*}
z \mapsto D_{x_i}f_i(x_i,z)
\end{equation*}
\begin{equation*}
 (\text{respectively }x_i \mapsto D_{z}f_i(x_i,z)) 
\end{equation*}
is injective. 
We will \textit{always} assume the following hypotheses:
\begin{description} 
\item[H1]	For all $i$, $f_i$ is $C^2$ and non-degenerate.
\item[H2] For each $(x_1,x_2,...,x_m)$ the supremum in \eqref{hp} is attained by a \textit{unique} $z=\overline{z}(x_1,x_2,...,x_m) \in Z$. 
\item[H3] $B(x_1,x_2,...x_m):=\sum_{i=1}^m D^2_{zz}f_i(x_i,\overline{z}(x_1,x_2,...,x_m))$ is non-singular.
\item[H4] $f_1$ is $(x_1, z)$-twisted.
\end{description}

Note that condition \textbf{H2} defines a mapping $\overline{z}:M_1 \times M_2...\times M_m \rightarrow Z$, where $\overline{z}(x_1,x_2,...,x_m)$ is the unique maximizer in \eqref{hp}.

In addition, we will \textit{sometimes} assume 
\begin{description}
\item[H5] For all $i$, $f_i$ is $(z,x_i)$-twisted.
\end{description}

It was shown in \cite{P}, Proposition 3.3.1, that conditions \textbf{H1-H3} above imply that the surplus $b$ is  $C^2$ and the following formulae hold:

\begin{eqnarray}
D_{x_i}b(x_1,x_2,...,x_m) = D_{x_i}f_i(x_i,\overline{z}(x_1,x_2,....,x_m)) \label{b_x}\\
D_{x_i}\overline{z}(x_1,x_2,...,x_m)=-B^{-1}\cdot(x_1, x_2,...,x_m)D^2_{zx_i}f_i(x_i,\overline{z}(x_1,x_2,...,x_m)) \label{z_x}\\
 D^2_{x_ix_j}b(x_1,x_2,...,x_m)=\label{b_xij}\\
=-D^2_{x_iz} f_i(x_i,\overline{z}(x_1,x_2,....,x_m))\cdot B^{-1}(x_1, x_2,...,x_m)\cdot D^2_{zx_j}f_j(x_j,\overline{z}(x_1,x_2,....,x_m)), \text{ for } i \neq j.\nonumber
\end{eqnarray}

\subsection{Carlier and Ekeland's formulation of the matching problem}
Carlier and Ekeland developed an alternative formulation of the optimal matching problem which will be crucial in our argument \cite{CarEke}; we briefly review it here.  Given a Borel probability measure $\nu$ on $Z$, define 
\begin{equation} \label{trandist}
T_{f_i}(\nu, \mu_i) := \sup \int_{M_1 \times Z} f_i(x_i,z) d\pi.
\end{equation}
where the supremum is over all measures $\pi$ on $M_1 \times Z$ whose marginals are $\nu$ and $\mu_i$.  The multi-agent matching problem studied by Carlier and Ekeland is to find the $\nu$ which maximizes

\begin{equation}\tag{MAM}\label{MAM}
\nu \rightarrow \sum_{i=1}^m T_{f_i}(\nu, \mu_i) 
\end{equation}
We now record two results of Carlier and Ekeland which we will use in a crucial way.  First is the existence and uniqueness of a maximizer $\nu$ (Proposition 4 in \cite{CarEke}): 

\newtheorem{cesol}{Theorem}[subsection]
\begin{cesol}\label{cesol}
Assume conditions \textbf{H1} and \textbf{H4} and that $\mu_1$ is absolutely continuous with respect to Lebesgue measure.  Then there exists a unique maximizer $\nu$ for \eqref{MAM}.
\end{cesol}

The second result of Carlier and Ekeland which we will need encodes a natural relationship between the the problems \eqref{MAM} and \eqref{MK} (Proposition 3 in \cite{CarEke}):

\newtheorem{equivalent}[cesol]{Theorem}
\begin{equivalent}\label{equivalent}
Assume conditions \textbf{H1}, \textbf{H2} and \textbf{H4}.  Then, for the surplus $b$ defined by \eqref{hp}, the following hold.
\begin{enumerate}
\item If $\gamma$ solves \eqref{MK}, then $\overline{z}_{\#}\gamma = \nu$ maximizes \eqref{MAM}.
\item If $\nu$ maximizes \eqref{MAM}, then there exists a solution $\gamma$ to \eqref{MK} such that $\overline{z}_{\#}\gamma = \nu$.
\end{enumerate}
\end{equivalent}

\subsection{Rectifiability of the optimizer}

We will also need the following result, which is a special case of Theorem 2.3 in \cite{P}; see also Proposition 3.3.1 in \cite{P} for more details.
\newtheorem{spacelike}{Theorem}[subsection]
\begin{spacelike}\label{spacelike}
Assume \textbf{H1-H3}. Let $\gamma$ be any minimizer for \eqref{MK} and assume conditions \textbf{H1-H3}.  Then the support of $\gamma$ is $n$-rectifiable.  In addition, it is spacelike for the symmetric bilinear form 
\begin{equation*}
\sum_{i\geq2}[D^2_{x_1x_i}b + D^2_{x_ix_1}b].
\end{equation*}
That is, for any tangent vector $v=(v_1,v_2,...,v_m)$ to spt$(\gamma)$, we have:

\begin{equation*}
\sum_{i\geq2}v_1^T\cdot D^2_{x_1x_i}b\cdot v_i \geq 0.
\end{equation*}
\end{spacelike}

In view of the preceding result, we will let $t \in T$, where  
\begin{equation}
T:=spt(\gamma) \subseteq M_1 \times M_2 \times ...\times M_m, 
\end{equation}
and 
\begin{equation}\label{coordmaps}
x_1(t) \in M_1, x_2(t)\in M_2, ....,x_m(t) \in M_m
\end{equation}
will denote the coordinate maps.  Note that, as $x_1:T \rightarrow M_1$ is certainly Lipschitz, if $\mu_1$ is absolutely continuous in local coordinates, $\gamma$ must be absolutely continuous with respect to $n$ - dimensional Hausdorff measure $H^n$ on $T$.  Furthermore,, the maps $t \mapsto x_i(t)$ are all differentiable wherever the rectifiable set $T$ is smooth enough to differentiate, which is $H^n$  almost everywhere and therefore $\gamma$ almost everywhere.  Together with the spacelike condition in the preceding theorem, this implies that 

\begin{equation*}
\sum_{i\geq2}(D_tx_1(t))^T\cdot D^2_{x_1x_i}b(x_1(t),x_2(t),...,x_m(t))\cdot D_tx_i(t) \geq 0
\end{equation*}
$\gamma$ almost everywhere.
\section{Absolute continuity of $\nu$.}
In this section, we prove that the solution $\nu$ to \eqref{MAM} is absolutely continuous with respect to local coordinates.  This can be seen as a structural result on $\nu$, interesting in its own right, but it is also a key part of the machinery needed to prove uniqueness and structural results for \eqref{MK}.  

\newtheorem{gammanull}{Lemma}[section]
\begin{gammanull}
Assume conditions \textbf{H1-H4} and that $\mu_1$ is absolutely continuous with respect to Lebesgue measure.  Then, for any solution $\gamma$ to \eqref{MK}, 
\begin{equation*}
\gamma[\{t: |det(D_tx_1(t))|=0  \}]=0,
\end{equation*}
where $x_1(t)$ is defined by \eqref{coordmaps}.
\end{gammanull}
\begin{proof}
Set $A = \{t: |det(D_tx_1(t))|=0  \}$. We first show that the image of $A$ under $x_1$,  
\begin{equation*}
x_1(A) = \{x_1(t):|D_tx_1(t)|=0 \}
\end{equation*}
has Lebesgue measure zero; this is a simple application of the Lipschitz area formula.  Indeed, the coarea formula tells us that

\begin{eqnarray*}
\int_A|detD_tx_1(t)|dH^n(t) = \int_{M_1}\#[A \cap x_1^{-1}\{y\}]dH^n(y)=\int_{x_1(A)}\#[A \cap x_1^{-1}\{y\}]dH^n(y).
\end{eqnarray*}

Here, $H^n$ denotes $n$-dimensional Hausdorff measure on both $T$ and $M_1 \subseteq \mathbb{R}^n$ (of course, in the latter case, this coincides with Lebesgue measure).  As $|detD_tx_1(t)| =0$ on $A$ by definition, the left hand side above is equal to zero.  On the other hand, $\#[A \cap x_1^{-1}\{y\} ]\geq 1$ for all $y \in x_1(A)$, and so 

\begin{eqnarray*}
\int_{x_1(A)}\#[A \cap x_1^{-1}\{y\}]dH^n(y) \geq \int_{x_1(A)} 1 dH^n(y).
\end{eqnarray*}

The right hand side above is simply the Lebesgue measure $|x_1(A)|$ of $x_1(A)$, and so combining the above two equations yields $|x_1(A)|=0$.

By absolute continuity, $\mu_1(x_1(A))=0$.  Now, note that 
\begin{equation*}
A=\{t: det|D_tx_1(t)|=0  \} \subseteq x_1^{-1}(x_1(A))=x_1^{-1}(\{x_1(t):det|D_tx_1(t)|=0 \}),
\end{equation*} 
and therefore
\begin{eqnarray*}
\gamma (A)&\leq& \gamma (x_1^{-1}(x_1(A))) \\
&=&\mu_1 (x_1(A))\\
&=&0.
\end{eqnarray*}
\end{proof}

\newtheorem{zinv}[gammanull]{Lemma}
\begin{zinv}\label{zinv}
Assume conditions \textbf{H1}-\textbf{H5}.  Then the mapping $T \ni t \mapsto \overline{z}(x_1(t),x_2(t),...,x_m(t)) \in Z$ has an invertible derivative almost everywhere, with respect to $n$-dimensional Hausdorff measure on $T$.
\end{zinv}
\begin{proof}
Wherever $(x_1(t),...,x_m(t))$ is differentiable (which is almost everywhere, by Theorem \ref{spacelike}) we have, by the chain rule and the derivative formulae \eqref{z_x}:
\begin{eqnarray*}
D_tz(x_1(t),....x_m(t)) &= &\sum_{i=1}^mD_{x_i}z(x_1(t),...,x_m(t))\cdot D_tx_i(t)\\
&=&-\sum_{i=1}^m  [B(x_1,x_2,...,x_m)] ^{-1} \cdot D^2_{zx_i}f_i(x_i,\overline{z}(x_1.x_2,...,x_m)) \cdot D_tx_i,
\end{eqnarray*}
where we have suppressed the argument $t$ in the last line.  Multiplying by 
\begin{equation*}
\big[D_tx_1(t)\big]^T\cdot\big[ D^2_{x_1z}f_1(x_1(t),\overline{z}(x_1(t),x_2(t),...,x_m(t)))\big],
\end{equation*}
and suppressing the arguments of all the functions, we have

\begin{eqnarray}
&&(D_tx_1)^T \cdot D^2_{x_1z}f_1 \cdot D_t\overline{z} = -\sum_{i=1}^m(D_tx_1)^T \cdot D^2_{x_1z}f_1 \cdot B^{-1}\cdot D^2_{zx_i}f_i\cdot D_tx_i\nonumber \\
&&=-(D_tx_1)^T \cdot D^2_{x_1z}f_1 \cdot B^{-1}\cdot D^2_{x_1z}f_1\cdot  D_tx_1 -\sum_{i=2}^mD_tx_1^T\cdot D^2_{x_1z}f_1\cdot B^{-1}\cdot D^2_{zx_i}f_i\cdot D_tx_i \nonumber \\
&&=-(D_tx_1)^T \cdot D^2_{x_1z}f_1 \cdot B^{-1}\cdot D^2_{zx_1}f_1\cdot D_tx_1 +\sum_{i=2}^mD_tx_1^T\cdot D^2_{x_1x_i}b\cdot D_tx_i \nonumber\\\label{D_tz}
\end{eqnarray}
Note that we have used formula \eqref{b_xij}.

Now, by the maximality of $z \mapsto \sum_{i=1}f_i(x_i,z)$ at $\overline{z}(x_1,x_2,...,x_m)$, $B(x_1,x_2,..,x_m):= \sum_{i=1}D^2_{zz}f_i(x_i,z)$ is negative semi-definite; by assumption \textbf{H3} it must therefore be negative definite.  Therefore, at every point where $|D_tx_1(t)| \neq 0$,  
\begin{equation*}
-(D_tx_1)^T \cdot D^2_{x_1z}f_1 \cdot B^{-1}\cdot D^2_{zx_1}f_1\cdot D_tx_1
\end{equation*}
is positive definite.  As the second term in \eqref{D_tz} is positive semi-definite by  Theorem \ref{spacelike}, the left hand side 
\begin{equation*}
D_tx_1^T \cdot D^2_{x_1z}f_1 \cdot D_t\overline{z}
\end{equation*}
must be positive definite and hence invertible.  Therefore, $D_t\overline{z}$ must be invertible, as desired.
\end{proof}

As we recalled in Theorem \ref{cesol}, Carlier and Ekeland proved that there exists a unique optimizer $\nu$ to \eqref{MAM} under the conditions  \textbf{H1} and \textbf{H4} \cite{CarEke}.  We prove below the main result of this section; that under conditions \textbf{H1}-\textbf{H5}, this optimizer is absolutely continuous in local coordinates.

\newtheorem{reg}[gammanull]{Theorem}

\begin{reg}\label{reg}
Assume $\mu_1$ is absolutely continuous with respect to Lebesgue measure.  Under conditions \textbf{H1}-\textbf{H5}, the maximizer $\nu$ to \eqref{MAM} is absolutely continuous with respect to Lebesgue measure.
\end{reg}

\begin{proof}
By Theorem \ref{equivalent}, we have $\nu = \overline{z}_{\#}\gamma$.  Now, Theorem \ref{spacelike} implies that
\begin{equation*}
\gamma =:h(t)dt
\end{equation*}
is concentrated on the $n$-rectifiable set $T$, and is absolutely continuous with respect to $n$-dimensional Hausdorff measure (see the discussion in subsection 2.3); we denote its density by $h$.  Furthermore, the map $T \ni t \mapsto \overline{z}(x_1(t),x_2(t),...,x_m(t)) \in Z$ is differentiable with an invertible derivative almost everywhere, by Lemma \ref{zinv}.  

Now, choose a set $A \subseteq Z$, negligible in local coordinates.  We need to show $\nu(A)=0$, or, equivalently $\gamma(\overline{z}^{-1}(A)) =0$.  Let $g:T \rightarrow \mathbb{R}$ be the characteristic function of the $\overline{z}^{-1}(A)$ and set 
\begin{equation*}
f(t) =g(t)h(t)/|det(\overline{z}(x_1(t),x_2(t),...,x_m(t)))|.
\end{equation*}
Then the co-area formula applied to $f$ yields:
\begin{eqnarray*}
\nu(A)&=&\gamma(\overline{z}^{-1}(A))\\
&=&\int_{T}g(t)h(t)dH^n(t) \\
&=&\int_{T}f(t)|detD_t\overline{z}(x_1(t),x_2(t),...,x_m(t))|dH^n(t) \\
&=&\int_{Z}\sum_{\overline{z}^{-1}(z)}f(t)dH^n(z)\\
& =& \int_{A}\sum_{\overline{z}^{-1}(z)}f(t)dH^n(z) +\int_{Z \setminus A }\sum_{\overline{z}^{-1}(z)}f(t)dH^n(z)
\end{eqnarray*}
The first term above is zero because $|A| =0$. We now show that the second term is zero as well.

Now, note that if $z \notin A$, then $\overline{z}^{-1}\{z\}$ does not intersect $\overline{z}^{-1}(A)$.  Therefore, for all $t \in \overline{z}^{-1}\{z\}$, $g(t)=0$ and therefore, for almost all $t \in \overline{z}^{-1}\{z\}$, $f(t)=0$.   It follows that $f(t) =0$ for all $t \in \overline{z}^{-1}\{z\}$, for $z \in Z \setminus A$, and therefore the integrand in the second term above is zero.  This completes the proof.
\end{proof}
\section{Monge solutions and uniqueness in the multi-marginal problem}
In this section we prove our main result; namely, that the optimizer $\gamma$ for \eqref{MK} is unique and induces a Monge solution.

\newtheorem{main}{Theorem}[section]

\begin{main}\label{main}
Assume $\mu_1$ is absolutely continuous with respect to Lebesgue measure.  Under conditions \textbf{H1-H5}, there exist Borel maps $F_i: Z \rightarrow X_i$ such that any optimizer $\gamma$ for \eqref{MK} is equal to the pushforward $(F_1,F_2,...,F_m)_{\#} \nu$, where $\nu$ is the unique optimizer for \eqref{MAM}. 
\end{main}
\begin{proof}
By Theorem \ref{reg}, $\nu$ is absolutely continuous with respect to Lebesgue measure and by \textbf{H5}, $f_i$ is $(z,x_i)$-twisted.  Note also that as the measure $\gamma$ on the compact set $M_1 \times M_2 \times...\times M_m$ has compact support and $\overline{z}$ is continuous, $\nu = \overline{z}_{\#} \gamma$ has compact support as well.  It is a well known result in optimal transport theory that these two conditions imply that each optimal coupling $\gamma_i$ in \eqref{trandist} is unique and is concentrated on the graph $\{z,F_i(z)\}$ of a unique mapping $F_i:Z \rightarrow M_i$.  Set 
\begin{equation*}
\tilde{\gamma} = (F_1,F_2,F_3,...,F_m)_{\#}\nu;
\end{equation*} 
we will show $\tilde{\gamma} = \gamma$.

Carlier and Ekeland proved that for any optimal $\gamma$, we have $\overline{z}_{\#} \gamma = \nu$ (Theorem \ref{equivalent}).  They also proved that the measure 
\begin{equation*}
(\pi_i,\overline{z}(x_1,x_2,...,x_m))_{\#}\gamma
\end{equation*}
on $M_i \times Z$ is optimal in \eqref{trandist}, where $\pi_i:M_1 \times M_2 \times...\times M_m \rightarrow M_i$ is the canonical projection.   It follows from the remarks above that 
\begin{equation*}
(\pi_i,\overline{z}(x_1,x_2,...,x_m))_{\#}\gamma =\gamma_i = (F_i, Id)_{\#}\nu.
\end{equation*}
Note that this implies that, $\gamma$ almost everywhere, we have 
\begin{equation*}
x_i=    \pi_i(x_1,x_2,...,x_m) = F_i(\overline{z}(x_1,x_2,...,x_m)).
\end{equation*}
    Therefore, $\gamma$ almost everywhere we have 
\begin{eqnarray*}
(x_1,x_2,....,x_m)& =& \Big(F_1\big(\overline{z}(x_1,x_2,...,x_m)\big), F_2\big(\overline{z}(x_1,x_2,...,x_m)\big),...,F_m\big(\overline{z}(x_1,x_2,...,x_m)\big)\Big)\\
&:=&F(\overline{z}(x_1,x_2,...,x_m)),
\end{eqnarray*}
where we have defined $F:Z \rightarrow M_1 \times M_2 \times ...\times M_m$ as $F:=(F_1,F_2,...,F_m)$.  Now, given any continuous function $H:M_1 \times M_2 \times...\times M_m \rightarrow \mathbb{R}$, this implies 
\begin{eqnarray*}
\int_{M_1 \times M_2\times ...\times M_m}H(x_1,...,x_m)d\gamma(x_1,...,x_m) &=&\int_{M_1 \times M_2\times ...\times M_m}H(F(\overline{z}(x_1,...,x_m)))d\gamma(x_1,...,x_m)\\
&=&\int_{Z}H(F(z))d\nu(z) \\
&=&\int_{M_1 \times M_2\times ...\times M_m}H(x_1,...,x_m)d\tilde{\gamma}(x_1,...,x_m)
\end{eqnarray*}
As $H \in C^0(M_1 \times M_2 \times...\times M_m)$ was arbitrary, this implies $\gamma = \tilde{\gamma}$, as desired.
\end{proof}

\newtheorem{maincor}[main]{Corollary}
\begin{maincor}\label{maincor}
Under the conditions in Theorem \ref{main}, the optimal measure $\gamma$ is unique. 
\end{maincor}
\begin{proof}
The result follows immediately from the uniqueness of $\nu$ and the maps $F_i$ in the preceding theorem.  The former is proven by Carlier and Ekeland (Theorem \ref{cesol} here), while the latter follows easily from the standard theory of optimal transportation.
\end{proof}

Finally, we show that the (unique) optimal measure $\gamma$ induces a Monge solution.
\newtheorem{monge}[main]{Corollary}
\begin{monge}\label{monge}
Under the conditions in Theorem \ref{main}, the optimal measure $\gamma$ has the form $(Id,G_2,G_3,...G_m)_{\#} \mu_1$, for unique measurable maps $G_i:M_1 \rightarrow M_i$.
\end{monge}
\begin{proof}
Setting $G_i = F_i \circ F_1^{-1}$, we easily obtain $ (Id,G_2,G_3,...G_m)_{\#} \mu_1=\gamma$.
\end{proof}
\newtheorem{econint}[main]{Remark }
\begin{econint}
\textbf{\emph{(economic interpretation)}} It seems appropriate to briefly discuss the economic interpretation of Theorem \ref{main}, in terms of the matching models in \cite{cmn} and \cite{CarEke}.  In particular, those two papers discuss slightly different notions of purity. 

 Carlier and Ekeland defined an equilibrium to be pure when workers of the same type necessarily signed contracts of the same type.  In mathematical terms, this means that for the optimal $\nu$ in \eqref{MAM}, each optimal coupling $\gamma_i$ between $\mu_i$ and $\nu$ is concentrated on the graph of a function $F_i: M_i \rightarrow Z$.  

Chiappori, McCann and Nesheim, on the other hand, defined an equilibrium in the $m=2$ case to be pure when buyers of the same type almost always conducted business with sellers of the same type. This means that the optimal measure $\gamma$ in $\eqref{MK}$ is concentrated on the graph of a function $G_2:M_1 \rightarrow M_2$. 

The purity notion of Chiappori, McCann and Nesheim extends naturally to the $m\geq 3$ setting.  We will say that an equilibrium is pure, in the sense of Chiappori, McCann and Nesheim, if the optimal $\gamma$ in \eqref{MK}  is concentrated on the graph of a function over $x_1$.  Economically, this means that workers in $M_1$ of the same type almost always conduct business with a team of the same type.

Purity of the equilibrium in the sense of Carlier and Ekeland is proven in \cite{CarEke}; Corollary \ref{main} can be interpreted as purity in the sense of Chiappori, McCann and Nesheim.
\end{econint}
\section{Examples and applications}

In this section we present several examples and applications of our main theorem.  The examples in the first subsection demonstrates that Theorem \ref{main} is a generalization of the seminal paper of Gangbo and \'Swi\c{e}ch \cite{GS}; in fact, we show that our work here implies a partial generalization of their result to manifolds.  The example in subsection 5.2 shows that Corollary \ref{monge} also encompasses a result of Heinich \cite{H}.  In subsection 5.3, we present an example of a surplus function which is \textit{not} of the form \eqref{hp}, but which satisfies different conditions required for uniqueness and Monge solutions in \cite{P1}; this verifies that Corollaries \ref{maincor} and \ref{monge} do not generalize the main theorem in \cite{P1}.  On the other hand, we exhibit an example in subsection 5.4 which \textit{is} of the form \eqref{hp}, but does not satisfy the conditions in \cite{P1}.  Therefore, the main theorem in \cite{P1} does not generalize Corollaries \ref{maincor} and \ref{monge} either. 

\subsection{Generalization of the two marginal quadratic cost on manifolds}
In this subsection, we consider the case where each $M_i \subseteq M$ is a bounded, open subset of some common Riemannian manifold $M$, $Z=M$ and the $-f_i$ are all the Riemannian distance squared.  In this case, we have:
\begin{equation}\label{gsman}
b(x_1,x_2,...,x_m) = \sup_{z \in Z}\sum_{i=1}^m -d^2(x_i,z)=\sup_{z \in M}\sum_{i=1}^m -d^2(x_i,z)
\end{equation}
Note that maximizing \eqref{MK} is equivalent to \textit{minimizing}  
\begin{equation*}
\int_{M_1 \times M_2 \times ...\times M_m} c(x_1,x_2,...,x_m) d\gamma (x_1,x_2,...,x_m)
\end{equation*}
over the same set of measures for the \textit{cost} function

\begin{equation}
c(x_1,x_2,...,x_m)=-b(x_1,x_2,...,x_m) = \inf_{z \in M}\sum_{i=1}^m d^2(x_i,z)
\end{equation}
Recall that $M$ is called a Hadamard manifold if it is complete, simply connected and has non-positive sectional curvature.  Our main result easily implies the following:

\newtheorem{riemmann}{Corollary}[subsection]
\begin{riemmann}\label{riemmann}
Let $M$ be a Hadamard manifold.  Let $\mu_1,\mu_2,...\mu_m$ be compactly supported Borel probability measures on $M$ and assume $\mu_1$ is absolutely continuous with respect to local coordinates.  Then the solution $\gamma$ to \eqref{MK} with surplus $b$  given by \eqref{gsman} is unique and is concentrated on a graph over $x_1$.
\end{riemmann}

\begin{proof}
By the Cartan-Hadamard theorem, $M$ is diffeomorphic to $\mathbb{R}^n$ and has no cut locus.  Conditions \textbf{H1,H4} and \textbf{H5} are then easily verified.  Furthermore, it is well known that each mapping
\begin{equation*}
z \mapsto d^2(x_i,z)
\end{equation*}
is uniformly convex along geodesics.  Therefore, 
\begin{equation*}
 z \mapsto \sum_{i=1}^m d^2(x_i,z)
\end{equation*}
is uniformly convex along geodesics, which implies conditions \textbf{H2} and \textbf{H3}.  The result now follows directly from Theorem \ref{main} and Corollaries \ref{maincor} and \ref{monge}.
\end{proof}

\newtheorem{gengs}[riemmann]{Remark}
\begin{gengs}Gangbo and \'Swi\c{e}ch proved uniqueness and existence of Monge solutions for the quadratic surplus \eqref{gss} on Euclidean space; a natural open question is whether this result can be generalized to Riemannian manifolds.  
Perhaps the most obvious generalization of this result would consist of uniqueness and Monge solution results for the cost function $\sum_{i \neq j}^md^2(x_i,x_j)^2$.  However, the preceding Corollary can be interpreted as a different type of generalization.  Note that 
\begin{eqnarray*}
\sup_{z \in \mathbb{R}^n}\sum_{i=1}^m-|x_i-z|^2 &=& -\sum_{i}|x_i-\frac{1}{m}\sum_j x_j|^2\\
&= & \sum_{i}(2 x_i \cdot \frac{1}{m}\sum_j x_j- |x_i|^2 -\frac{1}{m^2}|\sum_j x_j|^2)\\
&=&\sum_{i}(2 x_i \cdot \frac{1}{m}\sum_j x_j- |x_i|^2) -\frac{1}{m}|\sum_j x_j|^2\\
&=&\sum_{i}(2 x_i \cdot \frac{1}{m}\sum_j x_j- |x_i|^2) -\frac{1}{m}\sum_j\sum_k x_j\cdot x_k\\
&=&\frac{1}{m}\sum_{i} \sum_j x_i\cdot  x_j- \sum_i|x_i|^2\\
&=&- \frac{1}{2m}\sum_i\sum_j |x_i-x_j|^2
\end{eqnarray*} 
This establishes two facts:

\begin{enumerate}
\item The quadratic surplus \eqref{gss} is of the form \eqref{hp}, for $f_i(x_i,z) = -|x_i-z|^2$, which, as can easily be checked, satisfy \textbf{H1-H5}.  Corollaries \ref{maincor} and \ref{monge} therefore generalize the theorem of Gangbo and \'Swi\c{e}ch.
\item In particular, the preceding Corollary shows that our result implies a generalization of Gangbo and \'Swi\c{e}ch's theorem to Hadamard manifolds, 
\end{enumerate}
\end{gengs}
\newtheorem{genmccann}[riemmann]{Remark}
\begin{genmccann}
Brenier's celebrated polar factorization theorem amounts to uniqueness and existence of Monge solutions for the optimal transport problem on $M_1=M_2=\mathbb{R}^n$ with surplus $-|x-y|^2$ \cite{bren}.  McCann's generalization of this result resolves the same questions for the quadratic distance squared $-d^2(x_1,x_2)$ on a Riemannian manifold \cite{m3}.

Note that 
\begin{equation*}
\frac{-|x-y|^2}{2} = \sup_z(-|x-z|^2-|y-z|^2)
\end{equation*}
 and that on a Riemannian manifold 
 \begin{equation*}
 \frac{-d(x,y)^2}{2} = \sup_z(-d(x,z)^2-d(y,z)^2).
 \end{equation*}
Therefore, the preceding Corollary can be seen as a multi-marginal version of McCann's theorem (albeit only for Hadamard manifolds).  Furthermore, by the previous remark, Corollary \ref{riemmann} generalizes the Gangbo and \'Swi\c{e}ch result in the same spirit that McCann's theorem generalizes Brenier's theorem.
\end{genmccann}

To summarize the last two remarks: the surplus' of Brenier, McCann and Gangbo and \'Swi\c{e}ch are all of the form \eqref{hp} with quadratic functions $f_i$.   For Hadamard manifolds, then, Corollary \ref{riemmann} can therefore be thought of as generalizing McCann's surplus to several marginals, or the Gangbo-\'Swi\c{e}ch surplus to Riemannian manifolds.  One should note that, unlike in \cite{m3}, we deal here exclusively with a class of manifolds which are diffeomorphic to $\mathbb{R}^n$ and therefore avoid entirely complications arising from the cut locus, including non-smoothness of the $f_i$ and non-uniqueness of the minimizing $\overline{z}(x_1,x_2,...,x_m)$.  On the other hand, to this point very little is known about multi-marginal problems on \textit{any} non-flat manifolds with surplus functions derived from the distance, and so I feel Corollary \ref{riemmann} is an interesting contribution in this direction.

Finally, recall that the barycenter (with equal weights) of several points $p_1,p_2,...,p_m$ in a metric space is the minimizer of  $q \mapsto \sum_{i=1}^m d^2(p_i,q)$, assuming this minimum exists uniquely.  Recall that Wasserstein space over a Riemannian manifold $M$, the space of Borel probability measures on $M$ endowed with the Wasserstein metric

\begin{equation*}
W_{\frac{d^2(\cdot)}{2}}(\mu_1,\nu):= \sqrt{\inf \int_{M^2} \frac{d^2(x,z)}{2} d\pi}
\end{equation*}
where the infimum is over all measure $\pi$ on $M^2$ whose marginals are $\mu$ and $\nu$, is a metric space.

For the surplus function \eqref{gsman}, the optimizer $\nu$ in \eqref{MAM} in this case coincides with the barycenter of the measures $\mu_1,\mu_2,...\mu_m$.  Barycenters in Wasserstein space over $\mathbb{R}^n$ were studied by Carlier and Agueh, who proved existence, uniqueness and regularity results.  The present author proved existence and uniqueness of barycenters in Wasserstein space over Riemannian manifolds in \cite{P7}; the following result follows immediately from Theorem \ref{reg} and can be seen as a generalization of the regularity result of Agueh and Carlier.
\newtheorem{riemannbc}[riemmann]{Corollary}
\begin{riemannbc}
Let $\mu_1,\mu_2,...\mu_m$ be compactly supported Borel probability measures on a Hadamard manifold $M$ and assume $\mu_1$ is absolutely continuous with respect to local coordinates.  Then the barycenter $\nu$ is absolutely continuous with respect to local coordinates.
\end{riemannbc}

\subsection{Convex functions of the sum}
Let $b(x_1,x_2,...,x_m) = h(\sum_{i=1}^mx_i)$, where $h:\mathbb{R}^n \rightarrow \mathbb{R}$ is smooth and uniformly convex.  Letting $h^*$ denote the Legendre transform, we have that 

\begin{equation*}
h(\sum_{i=1}^mx_i) = \sup_{z \in \mathbb{R}^n}\sum_{i=1}^m x_i \cdot z -h^*(z)
\end{equation*}
Letting $f_1(x_1,z) = x_1\cdot z - h^*(z)$ and $f_i(x_i,z) = x_i\cdot z$ for $i \geq 2$, we have that $b$ is of the form $\eqref{hp}$.  It is straightforward to verify the conditions \textbf{H1-H5}, and so our main theorem generalizes the result of Heinich \cite{H}, at least under smoothness and uniform convexity conditions (the result of Heinich required only strict convexity of $h$).

\subsection{A non hedonic pricing surplus}
In the last two subsections, we have seen that our main result generalizes the results of Gangbo and \'Swi\c{e}ch and Heinich.  It is natural to wonder whether it also encompasses the result in \cite{P1}.  Below, we answer this question negatively, by exhibiting a surplus function satisfying the conditions in \cite{P1} which is not of the form \eqref{hp}.

Let $m=3$ and set 
\begin{equation}\label{bilinear}
b(x_1,x_2,x_3) = x_1 \cdot x_2 +x_1 \cdot x_3 + x_2 \cdot Ax_3
\end{equation}
where $A$ is an $n \times n$ matrix which is positive definite but \textit{not symmetric}.  It was verified in \cite{P1} (Example 4.2) that this surplus function satisfies the conditions therein.  We now show that this surplus function cannot be of the form \eqref{hp}. 

\newtheorem{nothp}{Proposition}[subsection]

\begin{nothp}
The surplus $b$ defined by \eqref{bilinear} is not of the form \eqref{hp} for any functions $f_1,f_2$ and $f_3$ satisfying \textbf{H1-H6}.
\end{nothp}
\begin{proof}
The proof is by contradiction; assume there exist functions $f_1,f_2$ and $f_3$ satisfying \textbf{H1-H6} so that $b = \sup_{z \in Z}\sum_{i=1}^3f_i(x_i,z)$.  We compute the following product of matrices of mixed partials:

\begin{equation*}
D^2_{x_2x_3}b \cdot [D^2_{x_1x_3}b]^{-1} \cdot D^2_{x_1x_2}b
\end{equation*}
On  one hand, by \eqref{bilinear}, this is 
\begin{equation*}
D^2_{x_2x_3}b \cdot [D^2_{x_1x_3}b]^{-1} \cdot D^2_{x_1x_2}b=A.
\end{equation*}
On the other hand, using $\eqref{b_xij}$ and suppressing the arguments of the functions, we have
\begin{eqnarray*}
&&D^2_{x_2x_3}b\cdot  [D^2_{x_1x_3}b]^{-1}\cdot  D^2_{x_1x_2}b\\
&=&-D^2_{x_2z}f_2\cdot B^{-1}\cdot D^2_{zx_3}f_3\cdot [D^2_{x_1z}f_1\cdot B^{-1}\cdot D^2_{zx_3}f_3]^{-1}\cdot D^2_{x_1z}f_1\cdot B^{-1}\cdot D^2_{zx_2}f_2\\
&=&-D^2_{x_2z}f_2\cdot B^{-1}\cdot D^2_{zx_3}f_3\cdot[D^2_{zx_3}f_3]^{-1}\cdot B \cdot[D^2_{x_1z}f_1]^{-1}\cdot D^2_{x_1z}f_1\cdot B^{-1}\cdot D^2_{zx_2}f_2\\
&=&-D^2_{x_2z}f_2 \cdot B^{-1}\cdot D^2_{zx_2}f_2
\end{eqnarray*}
We therefore have 
\begin{equation}\label{cont}
A=-D^2_{x_2z} f_2\cdot B^{-1}\cdot D^2_{zx_2}f_2
\end{equation}
As $B^{-1}$ is symmetric and $D^2_{zx_2}f_2 = [D^2_{x_2z}f_2]^{T}$, $-D^2_{x_2z}f_2\cdot B^{-1}\cdot D^2_{zx_2}f_2$ is symmetric.  On the other hand $A$ is not symmetric, and so \eqref{cont} yields a contradiction.  This completes the proof.
\end{proof}

\subsection{Suprema of convex functions}
We now consider the case where each $f_i$ is a concave function of the sum.
\newtheorem{ConSum}{Corollary}[subsection]
\begin{ConSum}
Assume for all $i$ that $f_i(x_i,z) = h_i(x_i+z)$ for a smooth, uniformly concave function $h_i$ such that $\lim_{|y| \rightarrow \infty}h_i(y) =-\infty$ and that the surplus is given by \eqref{hp}. If $\mu_1$ is absolutely continuous with respect to Lebesgue measure, then the optimal measure $\gamma$ for \eqref{MK} induces a Monge solution and is unique.
\end{ConSum}
\begin{proof}
We need only to verify conditions \textbf{H1-H5}, in order to apply Corollaries \ref{maincor} and \ref{monge}.  \textbf{H1, H3, H4} and \textbf{H5} follow from the smoothness and uniform convexity of $h_i$, while the growth condition ensures the attainment of a maximum (condition \textbf{H2}).  Therefore, our main theorem applies.
\end{proof}

Recall that in \cite{P1}, a strong second order condition on the surplus, called condition (III) in Theorem 3.1 there, is required to prove Monge solution and uniqueness results.  Next, we show that, for a surplus of the form \eqref{hp}, for well chosen $f_i(x_i,z)=h_i(x_i+z)$ (satisfying the hypotheses in the preceding Corollary), the condition (III) \textit{fails}, and so our result here is not strictly weaker than the theorem in \cite{P1}.

For simplicity, choose $m=3$ and take each $M_i$ to be a large open ball, centered at the origin.   In \cite{P1}, condition (III) was calculated explicitly for surpluses of form \eqref{hp} in Proposition 4.2.2;  amounts to the following:  

\emph{	For all choices $x_1,\tilde{x_1} \in M_1$, $x_2\in M_2$ and $x_3,\tilde{x_3} \in M_3$, the $n\times n$ matrix
\begin{eqnarray*}&&T_{x_1,x_2,x_3,\tilde{x_1},\tilde{x_3}}:=\\
&&-\Big[(D_{x_2z}f_2)B^{-1}(D^2_{zx_2}f_2)\Big]\Big(x_2,\overline{z}(x_1,x_2,x_3)\Big)+D^2_{x_2x_2}f_2\Big(x_2,\overline{z}(x_1,x_2,x_3)\Big)-D^2_{x_2x_2}f_2\Big(x_2,\overline{z}(\tilde{x_1},x_2,\tilde{x_3})\Big)
\end{eqnarray*}
is positive definite.}
 
We now show this condition fails for the choice $f_i(x_i,z) = -\sqrt{[1+|x_i +z|^2]}, i=1,2,3$.  This form of preference function was introduced in the optimal transport literature by Brenier \cite{BrenEMK}.  It also served as an early example of a cost function satisfying (A3w), the crucial condition introduced by Ma, Trudinger and Wang governing the regularity of optimal maps in two marginal problems \cite{mtw}.

It is easy to check that this function satisfies the conditions in the preceding Corollary, however, and so our surplus \eqref{hp} with these $f_i$ is an example of a surplus to which our main result here applies but the result in \cite{P1} does not. 
\newtheorem{fails}[ConSum]{Proposition}
\begin{fails}
 Condition (III) above fails for $b$ of the form \eqref{hp}, when $m=3$ and each $f_i(x_i,z) = -\sqrt{[1+|x_i +z|^2]}$.
\end{fails}
\begin{proof}
Take
\begin{eqnarray*}
x_1,x_2,x_3=0 && \tilde{x_3} =\tilde{x_1} = p,
\end{eqnarray*}
for some $p \in \mathbb{R}^n$ with $|p|=\frac{5}{2}$.  As $D_{z}f_i(x_i,z)=Dh_i(x_1+z)=-\frac{x_i+z}{[1+|x_i+z|^2]^{\frac{1}{2}}}$, we then have 

\begin{eqnarray*}
0&=&\sum_{i=1}^3D_zf_i(x_i,\overline{z}(x_1,x_2,x_3))\\
&=&\sum_{i=1}^3Dh(0+\overline{z}(0,0,0))\\
&=&-\frac{3\overline{z}(0,0,0)}{[1+|\overline{z}(0,0,0)|^2]^{\frac{1}{2}}}
\end{eqnarray*}
which implies 
\begin{equation*}
\overline{z}(x_1,x_2,x_3)=\overline{z}(0,0,0)=0.  
\end{equation*}

In a similar manner, we have
\begin{eqnarray*}
0&=&D_zf_1(\tilde{x_1},\overline{z}(\tilde{x_1},x_2,\tilde{x_3})) +D_zf_2(x_2,\overline{z}(\tilde{x_1},x_2,\tilde{x_3}))+D_zf_3(\tilde{x_3},\overline{z}(\tilde{x_1},x_2,\tilde{x_3})) \\
&=&Dh(p+\overline{z}(p,0,p))+Dh(0+\overline{z}(p,0,p))+Dh(p+\overline{z}(p,0,p))\\
&=&-2\frac{p+\overline{z}(p,0,p)}{[1+|p+\overline{z}(p,0,p)|^2]^{\frac{1}{2}}} -\frac{\overline{z}(p,0,p)}{[1+|\overline{z}(p,0,p)|^2]^{\frac{1}{2}}}
\end{eqnarray*}

One can easily verify (by, for example, direct substitution) that this yields 

\begin{equation*}
\overline{z}(p,0,p) =-\frac{4}{5}p.
\end{equation*}

In particular, as $|p|=\frac{5}{2}$, $|\overline{z}(p,0,p)|=2$.
Now note that 
\begin{equation*}
D^2h(y) = \frac{1}{[1+|y|^2]^{\frac{1}{2}}}(\frac{y\cdot y^T}{1+|y|^2} -I).
\end{equation*}
It is then straightforward to compute 
\begin{eqnarray*}
B(x_1,x_2,x_3) &= &B(0,0,0) =3D^h(0+0)=-3I\\
D^2_{x_2z}f_2(x_2,\overline{z}(x_1,x_2,x_3)&=&D^2h(0+0) =-I
\end{eqnarray*}
Therefore
\begin{eqnarray*}
\Big[(D^2_{x_2z}f_2)B^{-1}(D^2_{zx_2}f_2)\Big]\Big(x_2,\overline{z}(x_1,x_2,x_3)\Big) &=&\Big[(D^2_{x_2z}f_2)B^{-1}(D^2_{zx_2}f_2)\Big]\Big(0,0\Big)\\
&=&-\frac{1}{3}I
\end{eqnarray*}
Similarly, we have

\begin{eqnarray*}
D^2_{x_2x_2}f_2\Big(x_2,\overline{z}(x_1,x_2,x_3)\Big) &=&-I,\\
D^2_{x_2x_2}f_2\Big(x_2,\overline{z}(\tilde{x_1},x_2,\tilde{x_3})\Big)&=& \frac{1}{[1+|\overline{z}(p,0,p)|^2]^{\frac{1}{2}}}\Big(\frac{\overline{z}(p,0,p)\cdot\overline{z}(p,0,p)^T}{1+|\overline{z}(p,0,p)|^2} -I\Big)
\end{eqnarray*}
 We therefore have:
 
 \begin{eqnarray*}
 T_{x_1,x_2,x_3,\tilde{x_1},\tilde{x_3}}&=&\frac{1}{3}I -I-\frac{1}{[1+|\overline{z}(p,0,p)|^2]^{\frac{1}{2}}}(\frac{\overline{z}(p,0,p)\cdot\overline{z}(p,0,p)^T}{1+|\overline{z}(p,0,p)|^2} -I)\\
 &\leq&-\frac{2}{3}I+\frac{1}{[1+|\overline{z}(p,0,p)|^2]^{\frac{1}{2}}}I\\
 & =&[\frac{1}{[1+|\overline{z}(p,0,p)|^2]^{\frac{1}{2}}} -\frac{2}{3}]I\\
 &=&[\frac{1}{\sqrt{5}} -\frac{2}{3}]I
 \end{eqnarray*}
As $\frac{1}{\sqrt{5}} < \frac{2}{3}$, this is negative definite; in particular, it is not positive definite, so condition (III) fails.
\end{proof}
\bibliographystyle{plain}
\bibliography{biblio}
\end{document}